\documentclass[a4paper,12pt,intlimits,oneside]{amsart}
\usepackage{graphicx}
\usepackage{float}

\usepackage{enumerate}
\usepackage{amsfonts}
\usepackage{amsmath,amsthm,amssymb}

\newcommand{\comment}[1]{}

\newtheorem{theorem}{Theorem}

\newtheorem{proposition}[theorem]{Proposition}
\newtheorem{lemma}[theorem]{Lemma}

\newtheorem{remark}[theorem]{Remark}
\newtheorem{conjecture}[theorem]{Conjecture}

\newcommand\de{\delta}

\newcommand\dbj{\overline{d}_j}

\newcommand{\NN}{\mathbb N}
\newcommand{\ZZ}{\mathbb Z}
\newcommand{\RR}{\mathbb R}

\newcommand{\TT}{\mathbb T}

\newcounter{reb}
\newcounter{rev}
\newcounter{rem}
\newcounter{rec}
\begin{document}

\title[Counterexamples in the
Hardy-Littlewood majorant problem]{Three-term idempotent counterexamples in the
Hardy-Littlewood majorant problem}

\thanks{}

\author{S\'andor Krenedits}\thanks{Supported in part by the Hungarian National
Foundation for Scientific Research, Project \# K-81658.}

\date{\today}


\maketitle

\begin{abstract} The Hardy-Littlewood majorant problem was raised in the 30's and it can be formulated as the question whether $\int |f|^p\ge \int|g|^p$ whenever $\widehat{f}\ge|\widehat g|$. It has a
positive answer only for exponents $p$ which are even integers.
Montgomery conjectured that even among the idempotent polynomials
there must exist some counterexamples, i.e. there exists some
finite set of exponentials and some $\pm$ signs with which the
signed exponential sum has larger $p^{\rm th}$ norm than the
idempotent obtained with all the signs chosen $+$ in the
exponential sum. That conjecture was proved recently by
Mockenhaupt and Schlag. \comment{Their construction was used by
Bonami and R\'ev\'esz to find analogous examples among bivariate
idempotents, which were in turn used to show integral
concentration properties of univariate idempotents.}However, a
natural question is if even the classical $1+e^{2\pi i x} \pm
e^{2\pi i (k+2)x}$ three-term exponential sums, used for $p=3$ and
$k=1$ already by Hardy and Littlewood, should work in this
respect. That remained unproved, as the construction of
Mockenhaupt and Schlag works with four-term idempotents. We
investigate the sharpened question and show that at least in
certain cases there indeed exist three-term idempotent
counterexamples in the Hardy-Littlewood majorant problem; that is we have
for $0<p<6,  p \notin 2\NN$ $\int_0^{\frac12}|1+e^{2\pi ix}-e^{2\pi
i([\frac p2]+2)x}|^p > \int_0^{\frac12}|1+e^{2\pi ix}+e^{2\pi
i([\frac p2]+2)x}|^p$. The proof combines delicate calculus with numerical integration and precise error estimates.
\end{abstract}

\maketitle \vskip1em \noindent{\small \textbf{Mathematics Subject
Classification (2000):} Primary 42A05. \\[1em]
}



\section{Introduction}\label{sec:intro}

We denote, as usual, $\TT:=\RR/2\pi\ZZ$ the one dimensional torus
or circle group. Following Hardy and Litlewood \cite{HL}, $f$ is
said to be a majorant to $g$ if $|\widehat{g}|\leq \widehat{f}$.
Obviously, then $f$ is necessarily a positive definite function.
The (upper) majorization property (with constant 1) is the
statement that whenever  $f\in L^p(\TT)$ is a majorant of $g\in
L^p(\TT)$, then $\|g\|_p\leq \|f\|_p$. Hardy and Littlewood proved
this for all $p\in 2\NN$ -- this being an easy consequence of the
Parseval identity. On the other hand Hardy and Littlewood
observed that this fails for $p=3$. Indeed, they took
$f=1+e_1+e_3$ and $g=1-e_1+e_3$ (where here and in the sequel we
denote $e_k(x):=e(kx)$ and $e(t):=e^{2\pi i t}$, as usual) and
calculated that $\|f\|_3<\|g\|_3$.

The failure of the majorization property for $p\notin 2\NN$ was
shown by Boas \cite{Boas}. Boas' construction exploits Taylor
series expansion around zero: for $2k<p<2k+2$ the counterexample
is provided by the polynomials $f,g:=1+ re_1\pm r^{k+2}e_{k+2}$,
with $r$ sufficiently small to make the effect of the first terms
dominant over later, larger powers of $r$.

Utilizing an idea of Y. Katznelson, Bachelis proved \cite{Bac} the
failure of the majorization property for any $p\notin 2\NN$ even
with arbitrarily large constants. That is, not even $\|g\|_p < C_p
\|f\|_p$ holds with some fixed constant $C=C_p$.

For further comments and similar results in other groups see
\cite{Fou, Moc} .

Montgomery conjectured that the majorant property for $p\notin
2\NN$ fails also if we restrict to \emph{idempotent} majorants,
see \cite[p. 144]{Mon}. (A measure on an integrable function is
idempotent if its convolution square is itself: that is, if its
Fourier coefficients are either 0 or 1.) This has been recently
proved by Mockenhaupt and Schlag in \cite{MS}.

\begin{theorem}[{\bf Mockenhaupt \& Schlag}]\label{th:MMS} Let
$p>2$ and $p\notin 2\NN$, and let $k>p/2$ be arbitrary. Then for
the trigonometric polynomials $g:=(1+e_k)(1-e_{k+1})$ and
$f:=(1+e_k)(1+e_{k+1})$ we have $\|g\|_p >\|f\|_p$.
\end{theorem}

The quite nice, constructive example is given with a four-term
idempotent polynomial, although trinomials may seem simpler
objects to study. Indeed, there is a considerable
knowledge, even if usually for the maximum norm, on the space of
trinomials, see e.g. \cite{Cheb, Mini, Neu}. Note that striving for three-term examples is the absolute simplest we can ask for, as two-term polynomials can never exhibit failure of the majorization property.

In the construction of Mockenhaupt and Schlag, however, the key
role is played by the fact that the given 4-term idempotent
decomposes as the product of two two-term idempotents, which then
can be expressed by the usual trigonometric and hyperbolic
functions. So even if four term idempotents \emph{in general} are more complicated, than three term idempotents, but the \emph{particular product form} simplifies the analysis a great deal and gives way to a manageable calculation.

Nevertheless, one may feel that Boas' idea, i.e. the idea of
cancellation in the $(k+1)^{\rm st}$ Fourier coefficients works
even if $r$ is not that small -- perhaps even if $r=1$. The
difficulty here is that the binomial series expansion diverges,
and we have no explicit way to control the interplay of the
various terms occurring with the $\pm$ signed versions of our
polynomials. But at least there is one instance, the case of
$p=3$, when all this is explicitly known: already Hardy and
Littlewood \cite{HL} observed that failure of the majorant
property for $p=3$ is exhibited already by the pair of idempotents
$1+e_1\pm e_3$. In fact, this idempotent example led Montgomery to
express (in a vague form, however, see \cite{Mon}, p. 144) his conjecture on existence of \emph{idempotent} counterexamples.

There has been a number of attempts on the Montgomery problem. In
particular, Mockenhaupt has already addressed it fifteen years ago, see
\cite{Moc}, page 2 line 15. Moreover, that time Mockenhaupt worked in the range $2<p<4$ and exactly with the polynomials $1+e_1\pm e_3$, see also his footnote on p. 32. This attempt is based on an inequality (a discrete and uniform version of the inequality obtained by Hardy
and Littlewood only for the continuous case and $p=3$), which appears in Example 3.4 on p. 33 of \cite{Moc}, with a comment that "This lower bound is established by numerical calculations".

However, there is no convincing argument which would show that this hypothetical inequality would hold for all $p$, and so this preliminary attempt does not lead to a proof. In any case, we may say that Mockenhaupt expressed his view that $1+e_1\pm e_{k+2}$, where $2k<p<2k+2$, should provide a counterexample in the Hardy-Littlewood majorant problem, (at least for $k=1,2$). Our first aim is to analyze this question and execute proper numerical analysis to support this conjecture. In particular, we prove the assertion for $k=0,1,2$, justifying at least the cases which were concretely addressed by Mockenhaupt.

One motivation for us was the recent paper of Bonami and R\'ev\'esz \cite{AJM}.
In this breakthrough paper the authors settle a number of
questions regarding concentration of $p^{\rm th}$ integrals of
idempotents. In particular, they disprove a conjecture of
Anderson, Ash, Jones, Rider and Saffari, \cite{CRMany, Many} who
disbelieved concentration of idempotents for $p=1$. Also they
prove maximal concentration for all $p>0$ not an even integer (for
arbitrarily small open symmetric sets). Key to the construction of Bonami and R\'ev\'esz was the idea of constructing bivariate idempotents having special properties, related closely to the Hardy-Littlewood majorant problem. For details we refer to \cite{AJM}. It is also possible that their construction can be made simpler (work with less terms) by use of our methods here. To this question we hope to return in a later work.

The problem of idempotent polynomial concentration has its roots
in the analysis of weak-(2,2) type operators. For an account of
the topic from the origins to the present state of knowledge see
\cite{Ash1, Ash2}; see also \cite{Cow} for operator related matters and
\cite{DPQ, DPQ2, CRMany, Many} for development of the theme. Further questions of Wiener and Zygmund, which could be settled by the current strong results and methods of idempotent concentration, are discussed in \cite{CRAS}.

Relevance of idempotents can be well understood by the fact that
whenever a convolution operator represents a projection to a
finite dimensional translation-invariant subspace $H$ of say
$L^1(\TT)$, then $H$ is spanned by the exponentials in it, and
forming the idempotent $P_H:=\sum_{e_k\in H} e_k$ gives the
convolution kernel for the projection operator: $\Pi:
L^1(\TT)\rightarrow H$ is given by $\Pi f = f\star P_H$. In
particular, the Fourier partial sums operator $S_n$ is defined by
the Dirichlet kernel $D_n$ as convolution kernel. For more on this and the related famous Littlewood problem see e.g. \cite{Kony} and \cite{MPS}.

As already hinted by Mockenhaupt's thesis \cite{Moc}, proving that $1+e(x)\pm (e(k+2)x)$ would be a counterexample in the Hardy-Littlewood majorant problem may require some numerical analysis as well. However, we do not -- as we cannot -- pursue the numerical calculations outlined in \cite{Moc}. Instead, we do function calculus and support our analysis by numerical integration and error estimates where necessary. We are to discuss the following reasonably
documented conjecture.

\begin{conjecture}[]\label{conj:Moc} For all $p$ not an even integer, there are three-term idempotent counterexamples in the Hardy-Littlewood majorant problem.
\end{conjecture}

In fact, we address the more concrete form, going back to the examples of Hardy-Littlewood and Boas and discussed also by Mockenhaupt \cite{Moc}.

\begin{conjecture}\label{conj:con3} Let $2k<p<2k+2$, where $k\in \NN$ arbitrary.
Then the three-term idempotent polynomial $P_k:=1+e_1+e_{k+2}$ has
smaller $p$-norm than $Q_k:=1+e_1-e_{k+2}$.
\end{conjecture}

\section{Case $k=0$ of Conjecture \ref{conj:con3}}\label{sec:K0}

\begin{proposition}\label{prop:peakpol2} Let $F(x,y):=e(4y)+e(x+2y)+e(2x+y)$. Then,
for $p>2$, taking the marginal integral function
$f(y):=f_p(y):=\int_0^1 |F(x,y)|^pdx$, we have that (mod 1) $f$
has a unique, strict maximum at $0$. Conversely, for $0<p<2$ it
has strict global maximum at $\frac 12$.
\end{proposition}

\begin{remark} Note that $f_p(0)<f_p(1/2)$ for $0<p<2$ is exactly Conjecture \ref{conj:con3} for $k=0$.
\end{remark}
\begin{proof} (Based on the work \cite{AJM} of A. Bonami and Sz. Gy. R\'ev\'esz.) It is easy to see that $f$ is even: this comes from the identity $|F(-x,-y)|=|F(x,y)|$. Let us prove that it is monotonous on $[0,
\frac 12]$. Observe that
$$
|F(x+\frac{3y}{2},y)|=|e(4y)+e(x+\frac{7y}{2})+e(2x+4y)|=|2e(\frac y2)\cos (2\pi x) +1|.
$$
Now a translation of $x$ by 1/2 leads to a sign change of $\cos (2\pi x)$, therefore it
suffices to integrate $|2e(\frac y2)\cos (2\pi x) +1|^p$ on an interval of length 1/2, and
to add on the very same interval the integral of  $|-2e(\frac y2)\cos (2\pi x) +1|^p$. Thus
$$
f_p(y)=\int_{-\frac14}^{\frac14}\left(|2e(\frac y2)\cos (2\pi x)
+1|^p +|2e(\frac y2)\cos (2\pi x) -1|^p\right)dx.
$$
Any interval of length 1/2 would suffice, but we prefer to keep $\cos(2\pi x)$ positive,
otherwise there is a series of sign considerations which would make everything
overcomplicated: that suggests to choose $(-1/4,1/4)$.
So it is sufficient to show that the quantity
$$
\Phi(x,y):=|2e(\frac y2)\cos(2\pi x) +1|^p +|2e(\frac y2)\cos (2\pi
x) -1|^p
$$
is monotonous for $0<y<\frac 12$ and for fixed $x\in (-\frac 14,
\frac 14)$.

We take the derivative
\begin{align*}
\frac{\partial \Phi}{\partial y}(x,y)=&-2p\pi \sin(\pi y)\cos(2\pi x) \\ & \cdot \left\{
|2e(\frac y2)\cos(2\pi x) +1|^{p-2} - |2e(\frac y2)\cos (2\pi x)
-1|^{p-2} \right\}
\end{align*}

$2e(\frac y2)\cos(2\pi x)$ lies in the first quadrant, since
$e(\frac y2)=e^{\pi iy}$ lies there when $y\in (0, 1/2)$, and
$x\in (-1/4,1/4)$, so $\cos(2\pi x) >0$. Hence $|2e(\frac
y2)\cos(2\pi x) +1| > |2e(\frac y2)\cos(2\pi x) -1|$. We find that
the derivative's sign is the opposite of the sign of the
difference in the second line. It follows that $f_p$ has its
maximum at zero when $p>2$ and at 1/2 when $p<2$.
\end{proof}

\section{The $k=1$ case of Conjecture \ref{conj:con3}}\label{sec:k1}

\comment{\begin{proposition}\label{prop:k1xy} Let
$F(x,y):=1+e(x+y)+e(3x+2y)$ and consider the $p^{\rm th}$ marginal
integrals $f_p(y):=f(y):=\int_0^1 |F(x,y)|^p dx$. Then
for any $2<p<4$ 
we have $f(1/2)>f(0)$.
\end{proposition}

\begin{remark} Monotonicity of values of $f(y)$ is now not claimed!
\end{remark}
We are to compare
$$
f_p(0)=\int_{-1/2}^{1/2} |1+ e(x)+e(3x)|^p dx \quad \textrm{and} \quad f_p(1/2)=\int_{-1/2}^{1/2} |1- e(x)+e(3x)|^p dx.
$$
\bigskip

Observe that
$$
\int_{-1/2}^{1/2} |1 - e(x)+e(3x)|^p dx = \int_0^1 |1 - e(x-1/2)+e(3(x-1/2))|^p dx =
\int_0^1 |1 + e(x) - e(3x)|^p dx.
$$
}
To show the $k=1$ case of Conjecture \ref{conj:con3} it suffices to prove Proposition \ref{prop:k1x} below.

\begin{proposition}\label{prop:k1x} Let
$F_{\pm}(x):=1+ e(x)\pm e(3x)$ and consider the $p^{\rm th}$ marginal
integrals $f_{\pm}(p):=\int_0^1 |F_{\pm}(x)|^p dx$ as well as their difference $\Delta(p):=f_{-}(p)-f_{+}(p)=\int_0^1 |F_{-}(x)|^p-|F_{+}(x)|^p dx$. Then for all $p\in(2,4)$, $\Delta(p)>0$.
\end{proposition}
\begin{proof}
Let us introduce a few further notations. We will write $t:=p/2\in[1,2]$ and put
\begin{align}\label{eq:Gpmdef}
G_{\pm}(x)&:=|F_{\pm}(x)|^2,\qquad
g_{\pm}(t):=\frac12 f_{\pm}(2t)= \int_0^{1/2} G_{\pm}^t(x) dx,\qquad \\ \label{eq:ddef} d(t)&:=\frac 12 \Delta(2t)=g_{-}(t)-g_+(t)=\int_0^{1/2} \left[ G_{-}^t(x)-G_{+}^t(x)\right] dx.
\end{align}
Observe that $G$ being a nonnegative trigonometric polynomial, $d$ is an entire function of exponential type. So we are to prove that $d(t)>0$ for $1<t<2$. Note that by Parseval's formula $d(1)=d(2)=0$, since $2 \int_0^{1/2} G_{\pm}^1(x) dx=\int_0^1 F_{\pm}^2(x) dx=1+1+1$, and $2 \int_0^{1/2} G_{\pm}^2(x) dx=\int_0^1 F_{\pm}^4(x) dx=1+4+1+4+4+1=15$.

Our strategy in proving $d(t)>0$ will consist of two steps: first we prove that $d'(1)>0$, and then that $d'$ is concave in $[1,2]$, i.e. that $d^{'''}<0$. Since $d(1)=d(2)=0$, in view of Rolle's theorem $d'$ takes 0: but it can not have two different roots, as then by concavity at the endpoints of the interval $[1,2]$ it would have to assume negative values (while we will have $d'(1)>0$). Thus we find that $d'$ changes from positive to negative values at a unique zero point, say $\tau\in (1,2)$. It follows that $d$ increases between 1 and $\tau$ and decreases in $[\tau,2]$: so $\min_{[1,2]} d =\min \{ d(1), d(2)\}=0$, and $d$ is positive on $(1, 2)$.

\begin{lemma}\label{l:diffder1benpoz} We have $d'(1)>0$.
\end{lemma}
\begin{remark} By numerical calculation, $d'(1)\approx 0.0948...$, but we don't need the precise value. The only thing we need is that it is not too small, so allowing a feasible error bound for the approximate calculations, after deduction of a worst case error estimate the rest will still remain positive. Of course, to make our life as easy as possible, we set the error bound for the total error just below the already calculated numerical value. Therefore, preliminary numerical calculation of the value of $d'(1)$ only guides us in setting the parameters of the numerical proof, which in turn will prove positivity, but not the value of $d'(1)$.
\end{remark}

\begin{proof} We will give a detailed calculation, for it will serve as a model for the later, more general calculation with higher derivatives of $d$.

First of all observe that we have to consider the difference of two integrals, one for $G=G_{+}$ and another one for $G_{-}$, so writing
$$g^{'}(1):=g^{'}_{\pm}(1):=g^{(1)}_{\pm}(1):=\int_0^{1/2} G(x)\log G(x) dx,$$
we are to compute $d'(1)=g^{'}_-(1)-g^{'}_+(1)$.

Preliminary numerical calculation shows that finally we should find $d'(1)>0.09$, so for the two occurring numerical integration we may allow total errors up to 0.045, say.

We wish to use the standard approximation formulae\footnote{We
essentially could have $\|\Phi''\|_{1}$ etc. here.}
\begin{equation}\label{eq:Riemann}
\left| \int_0^{1/2} \Phi(\alpha) d\alpha - \frac{1}{2N}\sum_{n=1}^{N}
\Phi\left(\frac{n-1/2}{2N}\right)\right| \leq \min\left(
\frac{\|\Phi''\|_{\infty}}{192N^2},~\frac{\|\Phi'\|_{\infty}}{16N}\right),
\end{equation}
when numerically integrating $\Phi:=H:=G\log G$ along the $x$ values. As a first step, we compute the $x$-derivatives of $G(x)$ as
\begin{align}\label{eq:GpmdefFirst}
G_{\pm}(x)&=3+2\{\cos(2\pi x)\pm \cos(4\pi x)\pm \cos(6\pi x)\}\\
G_{\pm}^{'}(x)&=-4\pi \sin(2\pi x)\mp 8 \pi \sin(4\pi x)\mp 12 \pi \sin(6\pi x))\\
G_{\pm}^{''}(x)&=-8\pi ^2 \cos(2\pi x)\mp 32 \pi ^2 \cos(4\pi x)\mp 72 \pi ^2 \cos(6\pi x)).
\end{align}
Also we find
\begin{equation}\label{eq:Gpmnorm}
\|G_{\pm}\|_\infty \leq 9, \qquad \| G_{\pm}^{'}\|_\infty \leq 24 \pi, \qquad \|G_{\pm}^{''}\|_\infty \leq 112 \pi^2 .
\end{equation}

We also compare $G'$ and $\sqrt{G}=|F|$, more precisely $G'^2$ and $G$. (Note that
$G'= 2 |F| \cdot |F|'= 2 \sqrt{G}\cdot \left(\sqrt{G}\right)'$.) To this end we write $u=\cos v$ with $v=2\pi x$ and calculate
\begin{align*}
G_{\pm}(x)&=3+2\cos v \pm 2 \cos 2v \pm 2 \cos 3v = 3+2\cos v \pm 2 (2 \cos^2 v  -1 + 4 \cos^3 v -3 \cos v) \\ &= 3 + 2 u \pm 2(4 u^3 + 2 u^2 -3 u -1)=\begin{cases} 8u^3+4u^2-4u+1 & (G=G_+)\\ -8u^3-4u^2+8u+5 & (G=G_-)\end{cases}.
\end{align*}
Using these polynomial expressions in the range $|u|=|\cos(2\pi x)|\leq 1$, numerical calculation immediately gives
\begin{equation}\label{eq:Gpmminimum}
\min_{\TT} G_{+} \approx 0.3691... >1/e \qquad {\textrm{and}}\qquad \min_{\TT} G_{-} \approx 0.1249... > 1/9~\,.
\end{equation}
On the other hand
\begin{align*}
G'^2_{\pm}(x)&=(4\pi)^2 (\sin v\pm 2 \sin 2v \pm 3 \sin 3v)^2=(4\pi \sin v)^2 \left[1\pm  4 \cos v \pm 3( 4\cos^2v-1) \right]^2 \\ &= 16\pi^2(1-u^2)\left[(1\mp 3 \pm 4u \pm 12 u^2)\right]^2=\begin{cases} 64\pi^2 (1-u^2)(6 u^2 +2 u -1)^2 & (G=G_+)\\ 256\pi^2 (1-u^2)(3u^2+u-1)^2 & (G=G_-)\end{cases}.
\end{align*}
Therefore,
\begin{equation}\label{eq:GprimesquareperG}
\frac{G'^2_{\pm}}{G_{\pm}}(x)=\begin{cases} 64\pi^2 \frac{(1-u^2)(6 u^2 +2 u -1)^2}{8u^3+4u^2-4u+1} & (G=G_+)\\ 256\pi^2 \frac{(1-u^2)(3u^2+u-1)^2}{-8u^3-4u^2+8u+5} & (G=G_-)\end{cases}.
\end{equation}

These rational functions can be maximized numerically on the range $u\in[-1,1]$ of $u=\cos(2\pi x)$. We thus obtain
\begin{equation}\label{eq:GprimecompareG}
G'^2_{+} (x) < 1300 G_{+}(x) \qquad {\textrm{and}}\qquad G'^2_{-} (x) < 1100 G_{-}(x) .
\end{equation}
Similarly, we compare $G''$ and $G$, too. First, similarly as before
\begin{align*}
G^{''}_{\pm}(x)&=-8\pi^2 (\cos v\pm 4 \cos 2v \pm 9 \cos 3v)=-8 \pi^2 \left[\cos v \pm  4 (\cos^2 v -1) \pm 9 ( 4\cos^3 v-3 \cos v) \right] \\ &= -8 \pi^2\left[(u \pm 36 u^3 \pm 8 u^2 \mp 27 u \mp 4)\right]=\begin{cases} -8\pi^2 (36 u^3+8u^2-26 u -4) & (G=G_+)\\ 8 \pi^2 (36u^3 + 8u^2-28u-4) & (G=G_-)\end{cases}.
\end{align*}

Second, for the quotient we thus obtain
\begin{equation}\label{eq:GdoubleprimesquareperG}
\frac{G^{''}_{\pm}}{G_{\pm}}(x)=\begin{cases} -8 \pi^2 \frac{36 u^3+8u^2-26 u -4}{8u^3+4u^2-4u+1} & (G=G_+)\\ 8 \pi^2 \frac{36u^3 + 8u^2-28u-4}{-8u^3-4u^2+8u+5} & (G=G_-)\end{cases}.
\end{equation}
So finally numerical computation yields
\begin{equation}\label{eq:GdoubleprimecompareG}
|G^{''}_{+} (x)| < 2200 G_{+}(x) \qquad {\textrm{and}}\qquad |G^{''}_{-} (x)| < 4000 G_{-}(x) .
\end{equation}
Let us consider now the computation of $\Phi''(x)=H''(x)$, where $H:=G\log G$ with $G=G_{\pm}$. More generally, we can differentiate with respect to $x\in [0,1/2]$ the function $H(x):=H_{t,j,\pm}(x):=G^t(x)\log^j G(x)$, which we will need later. We get
\begin{align}\label{eq:Hprimeingeneral}
H'(x)=H'_{t,j,\pm}(x)&=\left\{ t G^{t-1}(x) \log^j G(x) + G^t(x) j \log^{j-1} (G(x)) \frac{1}{G(x)}\right\} G'(x)\notag \\& = G^{t-1}(x) G'(x) \log^{j-1} G(x) \left\{t\log G(x) + j \right\},
\end{align}
so in particular for $t=1$ and $k=1$ we conclude $H'(x)=H'_{1,1,\pm}(x)= G'(x) \left\{\log G(x) +1 \right\}$ and thus also
\begin{equation}\label{eq:Hdoubleprimespecial}
H''(x):=H^{''}_{1,1,\pm}(x)= G''(x) \left\{\log G(x) +1 \right\} +
\frac{G'^2(x)}{G(x)}.
\end{equation}
Therefore, we obtain from \eqref{eq:Gpmnorm}, \eqref{eq:Gpmminimum} and \eqref{eq:GprimecompareG} 
\begin{equation}\label{eq:Hdoubleprimespecial1}
\|H''\|_\infty:=\|H^{''}_{1,1,\pm}\|_\infty < 112\pi^2 \log(9e)+ 1300 < 4900.
\end{equation}
It follows that in the numerical integration formula \eqref{eq:Riemann} the step number could be chosen to satisfy $4900/(192N^2)<0.045$, that is $N>\sqrt{4900/8.64}\approx 23.81\dots$ i.e. $N\geq 24$.

We thus see that the Riemann sums of the form \eqref{eq:Riemann}
with $N\geq 24$ nodes will provide errors less than 0.045 in each
of the two integrals $g'_{\pm}(1)$, whence the total error in the
Riemann sum approximation of $d'(1)=g_-(1)-g_+(1)$ must lie below
0.09. On the other hand a standard numerical calculation of the
Riemann sums $g'_{\pm}(1)$ yields the approximate value $\approx
0.0948...$, which is well over $0.09$, hence the lemma is proved. (As for negligibility of the computational error occurring in the computer calculation of function values, see the more detailed analysis around formula \eqref{eq:Comperror}.)
\end{proof}


Note that from \eqref{eq:Hprimeingeneral} we can as well calculate the formula for $H''$ in the general case as
\begin{align}\label{eq:Hdoubleprimegeneral}
H''(x)&:=H^{''}_{t,j,\pm}(x) = G''(x) G^{t-1}(x) \log^{j-1} G(x) \left\{t \log G(x)+j \right\} \notag \\ & ~~+ G'^2(x) G^{t-2}(x) \log^{j-2} G(x) \left\{ t(t-1) \log^2 G(x)  + j(2t-1) \log G(x) + j(j-1) \right\}.
\end{align}

Our approach will be a computation of some approximating polynomial, which is, apart from a possible slight and well controlled error, a Taylor polynomial of $d'''$.

Numerical tabulation of values give that $d'''$ is decreasing from $d'''(1)\approx -0.2327...$ to even more negative values as $t$ increases from 1 to 2. Thus our goal is to set $n\in \NN$ and $\delta_j>0$, ($j=0,\dots,n+1)$ suitably so that in the Taylor expansion
\begin{equation}\label{eq:d3Taylor}
d^{'''}(t)=\sum_{j=0}^n \frac{d^{(j+3)}(\frac32)}{j!}\left(t-\frac32\right)^j +R_{n}(d^{'''},t),\qquad
R_{n}(d^{'''},t):=\frac{d^{(n+4)}(\xi)}{(n+1)!}\left(\xi-\frac32\right)^{n+1}
\end{equation}
the standard error estimate
\begin{align}\label{eq:Rd3t}
|R_n(d^{'''},t)|& \leq \frac{\|H_{\xi,n+4,+}\|_{L^1[0,1/2]} + \|H_{\xi,n+4,-}\|_{L^1[0,1/2]}}{(n+1)!} \cdot 2^{-(n+1)} \notag \\ &\leq  \frac{\frac12\|H_{\xi,n+4,+}\|_\infty + \frac12\|H_{\xi,n+4,-}\|_\infty }{(n+1)! 2^{n+1}}\\ & \leq \frac{\max_{1\leq \xi\leq 2} \|H_{\xi,n+4,+}\|_\infty + \max_{1\leq \xi\leq 2} \|H_{\xi,n+4,-}\|_\infty}{(n+1)! 2^{n+2}}.\notag
\end{align}
provides the appropriately small error $\|R_n(d^{'''},\cdot)\|_\infty <\delta_{n+1}$. Furthermore we want to compute appropriate approximation $\overline{d}_j$ of $d^{j+3}(3/2)$, such that
\begin{equation}\label{eq:djoverbarcriteria}
\left\|\frac{d^{(j+3)}(\frac32)-\overline{d}_j}{j!}\left(t-\frac32\right)^j\right\|_\infty =\frac{\left|d^{(j+3)}(\frac32)-\overline{d}_j\right|}{2^j j!}< \delta_j\qquad (j=0,1,\dots,n).
\end{equation}
Naturally, we wish to choose $n$ and the partial errors $\delta_j$ such that $\sum_{j=0}^{n+1}\de_j = \de:=0.231$, say, so that $d'''(t)< P_n(t) +\de$ with
\begin{equation}\label{eq:Pndef}
P_n(t):=\sum_{j=0}^n \frac{\overline{d}_j}{j!}\left(t-\frac32\right)^j.
\end{equation}
Here the approximate values $\dbj$ will be obtained by numerical integration, i.e. Riemann sums to approximate the integrals defining $d^{(j+3)}(3/2)$. Recall that
\begin{align}\label{eq:djintegral}
d^{(j)}(t)&=\frac{d^j}{d t^j} \left( \int_0^{1/2} \left[G_{-}^t(x) - G_{+}^t(x)\right]dx \right)\notag \\&= \int_0^{1/2} \left[\log^{j} G_{-}(x) \cdot G_{-}^t(x) \right]dx -\int_0^{1/2} \left[\log^{j} G_{+}(x) \cdot G_{+}^t(x)\right]dx=:g_{-}^{(j)}(t)-g_{+}^{(j)}(t).
\end{align}

To be precise, we apply the first error formula of
\eqref{eq:Riemann} with $N_j\in\NN$ steps, where $N_j$ are set in
function of a prescribed error of approximation $\eta_j$, which in
turn will be set in function of the choice of $\de_j$.

In fact, there is another source of error, that of the
computational error of the actual computer calculation of the
involved function values, used in computing the Riemann sums (to
approximate the integrals $g_{\pm}^{(j)}(t)$ in the formula \eqref{eq:djintegral} for
$d^{(j)}$). Let us agree that it is more than satisfactory to ensure a relative error bound of $10^{-4}$ for the total computational error as compared to the respective theoretical errors.

Let's denote the calculated value of a function $f$ by$f^*$. Then have to estimate
\begin{equation}\label{eq:Comperror}
\Delta_c:=|(G^t\log^jG)^*-G^t\log^jG|=|((G^t)^*-G^t)(\log^jG)^*+G^t((\log^jG)^*-\log^jG)|.
\end{equation}
We estimate this in parts. For the actual computation we
applied the MS Excel program, which computes the mathematical
functions with 15 significant digits of
precision\footnote{According to the user's manual, the MS Office
Excel 2003 program, what we have used throughout, calculates the
function values of the occurring mathematical functions with 15
significant digits of precision, see e.g.
{\texttt{http://office.microsoft.com/en-us/excel-help/
change-formula-recalculation-iteration-or-precision-HP010054149.aspx}}}.
$G$ both here in \eqref{eq:GpmdefFirst} for $k=1$ and later in
\eqref{eq:GpmdefFirst2} for the case $k=2$ consists of a sum of
cosine functions with coefficients $\pm2 $, so altogether with
weights $\le 6$. As $|\cos x| \le 1$, the error bound becomes $6
\times 0.5 \times 10^{-15}$, that is $|G^*-G| \le 3 \times
10^{-15}$.

Considering the cases $k=1$ and $k=2$, the values of $G$ always
lie between $1/16$ and $9$ in view of \eqref{eq:Gpmnorm} and
\eqref{eq:Gpmminimum} for $k=1$ and \eqref{eq:Gpmnorm2} and
\eqref{eq:Gpmminimum2} for $k=2$, respectively. That means that
the first significant digit of $\log G$ is at most at the place of
$10^0$, and the calculation error of the logarithm of it lies
below $0.5\times 10^{-14}$. Thus we can estimate
$|\log(G^*)-\log(G)|\le |G^*-G|\cdot \|\log'(G)\|_\infty< 3
\cdot10^{-15} \cdot 16$, and $|(\log(G^*))^*-\log(G)|<4.8 \times
10^{-14}+0.5 \times 10^{-14} < 5.3 \times 10^{-14}$.

To estimate the error in computing $G^t$ we write $G^t=e^{t\cdot
\log G}$, hence $|(G^t)^*-G^t|\leq G^t \times \left| \exp \left( t
\left\{ (\log G^*)^* -\log G\right\}\right)-1 \right|$. The
elementary estimate $|e^u-1|< 1.8 |u|$ when $|u|<0.5$ yields
$\left| \exp \left( t \left\{ (\log G^*)^* -\log
G\right\}\right)-1 \right| \leq 1.8 \times \left|t \left\{ (\log
G^*)^* -\log G\right\}\right| <  1.8 \times t \times 5.3 \times
10^{-14} < 10^{-13} t$ for all the possibly occurring values of
$t$ between $1$ and $3$. (Note that $t \times 5.3 \times
10^{-14}<0.5$.) In all, $|(G^t)^*-G^t| <10^{-13}\times G^t$.

Since $|\log G|\leq \log 16 < 2.8$ we have both for $|(\log
G^*)^*|$ and $|\log G|$ the upper estimate of 3. So for any $j\in
\NN$ we can write $|((\log G^*)^*)^j-\log^j G|\le|(\log
G^*)^*-\log G| \cdot \sum_{\ell=0}^{j-1} \left|((\log
G^*)^*)^{j-\ell}\log^{\ell} G \right|\le 5.3 \cdot 10^{-14}\times j
\times 3^{j-1}$ and also $|\log^j G|\leq 3^j$, $|((\log
G^*)^*)^j|<3^j$.

Turning to the estimation of \eqref{eq:Comperror} we thus get
$$
\Delta_c \leq 10^{-13}\times G^t \cdot 3^j + G^t \cdot 2 \cdot
10^{-14}\times j \times 3^j =(2j+10)\times 10^{-14} 3^j G^t < 32
\times 10^{-14} 3^{6+j}
$$
using that $1\leq t\leq 3$, $G \leq 9$, $j\leq 11$.

Note that actually we need $\Delta_c$ to be negligible compared to
$\eta_j$ (in Tables \ref{t:Deltas} and \ref{t:Deltas_2}) or the prescribed error $\delta$ (in Lemmas \ref{l:diffder1benpoz}, \ref{l:diffder2benpoz} and \ref{l:diff2der2benpoz}), the prescribed approximation error of the Riemann sums approximation of the occurring $d^{(j)}(t)$. (Observe that due to the Riemann sums approximation of various order derivatives, there is a shift of indices between the $j$ in the order of differentiation and the $j$ occurring in the formula $\eta_j=2^j j! \delta_j /2$ preceding \eqref{eq:Njchoice}). To meet the set relative error bound of $10^{-4}$, we want $\Delta_c<10^{-4} \eta_j$ or $\Delta_c<10^{-4} \delta$, respectively.

When $j$ is large, more precisely when $7\leq j \leq 11$, we just make a rough estimate of $\Delta_c$ using $j\leq 11$ only. This leads to $\Delta_c< 5\times 10^{-5}$. As $j\geq 7$ occurs only in Tables \ref{t:Deltas} and \ref{t:Deltas_2}, and the minimal value of such $\eta_j$-s in the two tables is $1.2$, we obtain $\eta_j\times 4.2 \times 10^{-5}>\Delta_c$.

In the case $j<7$ the smallest value is 0.025. Now we use value $j=6$, so $\Delta_c<1.8 \times 10^{-7}$, and $\eta_j\times 0.72 \times 10^{-5}>\Delta_c$

For $j\leq 6$, our choice of $\eta_j$ has a minimum of 0.025, while in the endpoint approximation lemmas (i.e. Lemmas \ref{l:diffder1benpoz}, \ref{l:diffder2benpoz} and \ref{l:diff2der2benpoz}) the minimal occurring $\delta$ is 0.017. In all we may use $j\leq 6$ and want $\Delta_c< 10^{-4} \times 0.017 =1.7 \times 10^{-6}$. However, substituting $j=6$ into the above error estimate yields $\Delta_c< 1.8 \times 10^{-7}$, which is

In cases of mentioned lemmas the value of $j$ is $1$ or $2$, so we use the $j=2$ for estimating $\Delta_c$, we obtain: $\Delta_c< 2.1 \times 10^{-9}$. The minimum of $\delta$-s in the three lemma is $0.034/2=0.017$ in Lemma \ref{l:diffder2benpoz}. It follows: $\delta \times 2.6 \times 10^{-7}>\Delta_c$.

So now we carry out this programme. First, as $G_{\pm}(x)\in
[1/9,9]$, $|\log^m G_{\pm}(x)|\leq 2^m \log^{m} 3$, and thus
$|H_{\xi,n+4,\pm}(x)|\leq 9^{\xi} 2^{n+4} \log^{n+4} 3\leq 81
\times 2^{n+4} \times (1.09861...)^{n+4}$, so setting
$\de_{n+1}=0.05$ with $n=7$ we find $\|R_7\|_\infty\leq 81 \times
8 \times (1.09861...)^{11} / 8! = 9 \times (1.09861...)^{11} /560
\approx 9 \times 2.8137... /560 \approx 0.04522... <\de_8=0.046$.

Now we must set $\de_0,\dots,\de_7$. The goal is that the termwise
error \eqref{eq:djoverbarcriteria} would not exceed $\de_j$, which
will be guaranteed by $N_j$ step Riemann sum approximation of the
two integrals defining $d^{(j+3)}(3/2)$ with prescribed error
$\eta_j$ each. Therefore, we set $\eta_j:=\de_j j!2^j/2$, and note
that

\begin{equation}\label{eq:Njchoice}
N_j > N_j^{\star}:=\sqrt{\frac{\|H^{''}_{3/2,j+3,\pm}\|_\infty}{192 \eta_j}} = \sqrt{\frac{\|H^{''}_{3/2,j+3,\pm}\|_\infty}{192 j! 2^{j-1} \de_j}}
\end{equation}
 suffices.
That is, we must estimate $\|H^{''}_{3/2,j,\pm}\|_\infty$ for
$j=3,\dots,10$ and thus find appropriate values of $N_j^{\star}$.

\begin{lemma}\label{l:H"norm} For $j=3,\dots,10$ we have the following numerical estimates for the values of $\|H^{''}_{3/2,j,\pm}\|_\infty$.
\begin{table}[H]
\caption{Estimated values of $\|H^{''}_{3/2,j,\pm}\|_\infty$ for $j=3,\dots,10$.}
\label{table:Hdoubleprimenorm}
\begin{center}
\begin{tabular}{cc}
$j$ \qquad & estimate for $\|H^{''}_{3/2,j,\pm}\|_\infty$ \\
3 & 195,745 \\
4 & 560,366 \\
5 & 1,577,686 \\
6 & 4,228,176 \\
7 & 11,254,403 \\
8 & 29,470,592 \\
9 & 76,110,084 \\
10 & 194,242,755 \\
\end{tabular}
\end{center}
\end{table}
\end{lemma}

\begin{proof}
\vbox{
Whether we consider $H_+$ or $H_-$, the range of $G_{\pm}$ stays in $[1/9,9]$, so
$$
\|H^{''}_{\pm}\|_\infty \leq \max\{\max_A |H^{''}_{3/2,j,\pm}|,~ \max_B |H^{''}_{3/2,j,\pm}| \}
$$
with $A:=\{x\in[0,1/2]~:~ G_{\pm}(x)\leq 1\}$ and $B:=\{x\in[0,1/2]~:~ G_{\pm}(x)>1\}$. Recall that from \eqref{eq:Hdoubleprimegeneral} we find for arbitrary $j\geq 3$ and with $G=G_{\pm}$
\begin{align*}
H^{''}_{3/2,j,\pm}= G''\sqrt{G} \log^{j-1}G \left(\frac32\log G +j\right) + \frac{G'^2}{\sqrt{G}} \log^{j-2}G\left(\frac34\log^2 G +2j \log G + j(j-1)\right).
\end{align*}
Since we have no control over the sign of $G''$, we now estimate trivially -- using \eqref{eq:Gpmnorm} -- as
\begin{align}\label{eq:HbyQ}
|H^{''}_{\pm}(x)| & \leq \|G''\| \sqrt{G(x)} \left|\log^{j-1}G(x) \right| \left|\frac32\log G(x) +j\right| + Q(x) \notag \\ &\left( \textrm{with} \quad Q(x):=\frac{G'^2(x)}{\sqrt{G(x)}} \left|\log^{j-2}G(x)\right|\left|\frac34\log^2 G(x) +2j \log G(x) + j(j-1)\right| \right) \notag \\ &\leq 112 \pi^2 3  \log^{j-1}9 (j+\log 27)+ Q(x).
\end{align}
}

Now by the two estimates of $|G'(x)|$ from \eqref{eq:Gpmnorm} $24 \pi$ and from \eqref{eq:GprimecompareG} $\sqrt{1300 G(x)}$ it follows that
$$
Q(x)\leq \begin{cases} 1300 G(x)^{1/2} \left|\log^{j-2}G(x)\right|\left|\frac34\log^2 G(x) +2j \log G(x) + j(j-1)\right| & x\in A \\ 576 \pi^2 G(x)^{-1/2} \log^{j-2}G(x) \left(\frac34\log^2 G(x) +2j \log G(x) + j(j-1)\right)  & x\in B
\end{cases}.
$$
Now observe that here $Q$ is estimated by functions of $G(x)$, so we can look for maximization or good estimates on the range of $G$. For $x\in A$ denote $u:=- \log G(x)$: then the condition $x\in A$ means that $0\leq u < \log 9$, while for $x\in B$ the substitution $u:=\log G(x)$ leads to $0<u \leq \log 9$. In all we find $\| Q\|_\infty \leq \max_{0\leq u \leq \log 9} \{\psi(u),~\varphi(u) \}$ with
$$
\psi(u):=1300 e^{-u/2} u^{j-2} \left| \frac34 u^2 - 2j u + j(j-1)\right|
$$
and
$$
\varphi(u) :=576 \pi^2 e^{-u/2} u^{j-2} \left( \frac34 u^2 + 2j u + j(j-1)\right)
$$
Now it is easy to observe that for any real $u\geq 0$ we have $| \frac34 u^2 - 2j u + j(j-1)| \leq \frac34 u^2 + 2j u + j(j-1)$ whence in view of $1300<576 \pi^2$, necessarily $\psi(u)\leq \varphi(u)$.

In all, $\|Q\|_\infty \leq \max_{0\leq u \leq \log 9} \varphi(u)$.
With a slight change of variable $v:=u/2$, we look for $576\pi^2
2^{j-2} \max_{[0,\log 3]} e^{-v}( 3v^j+4jv^{j-1}+j(j-1)v^{j-2})$.
The derivative of the function to be maximized is
$e^{-v}v^{j-3}p(v)$, where $p(v):=-3v^3-jv^2+3j(j-1)v
+j(j-1)(j-2)$. The first part is positive, and $p(v)$ is concave,
since $p'' (v)=-18v-2j <0$. Note that the concave function $p(v)$
starts with positive values as $p(0)=j(j-1)(j-2)>0$, and at
$+\infty$ it becomes negative, so $p(v)$ has at most 1 root, where
$p$ changes from positive values to negative ones. Consequently,
for any $j=3,4,...10$ the function $e^{-v}(
3v^j+4jv^{j-1}+j(j-1)v^{j-2})$ vanishes and increases at 0, then
it stays positive and tends to 0 at infinity, with one strict
maximum point in $(0,\infty)$ (at the single critical point where
its derivative vanishes); moreover, it is easy to see that it is
increasing for all $j\geq 3$ in the whole interval $[0,\log 3]$,
as there $p(v)$ stays positive. Therefore, the maximum is attained
at the right endpoint $v=\log 3$ of the interval, with maximum
values 94,948.95..., 303,717.77..., 916,480.8...,
2,649,475.04...,7,412,491.18..., 20,209,150.39...,
53,959,116.72..., 141,613,801.4 for $j=3,\dots,10$, respectively.
Adding $112 \pi^2\times 3  \log^{j-1}9 (j+\log 27)$ we get from
\eqref{eq:HbyQ} the numerical estimates of Table
\ref{table:Hdoubleprimenorm}.

\end{proof}


\begin{lemma}\label{l:d3at32} Set $\de_j$ as $\de_0=0.12$, $\de_1=0.040$, $\de_2=0.015$ $\de_3=0.004$, $\de_4=0.003$, $\de_5=\de_6=\de_7=0.001$. Then the approximate Riemann sums of order $N_j$ yield the approximate values $\overline{d}_j$ as listed in Table \ref{t:Deltas}, admitting the error estimates \eqref{eq:djoverbarcriteria} for $j=0,\dots,7$. Furthermore, $\|R_{7}(d^{'''},t)\|_{\infty} <0.046=:\de_8$ and thus with the approximate Taylor polynomial $P_7(t)$ defined in \eqref{eq:Pndef} the approximation $|d^{'''}(t)-P_7(t)|<\de$ holds uniformly for $1\leq t \leq 2$.

\begin{table}[H]
\caption{The chosen values of $\de_j$, $\eta_j$, the appropriate $N_j$ and approximate Taylor coefficients}
\label{t:Deltas}
\begin{center}
\begin{tabular}{|c|c|c|c|c|c|}
$j$ & $\de_j$ & $\eta_j$ & $N_j$ & $\overline{d}_j$\\
0 & 0.05 & 0.025 & 202 & -2.1079\\
1 & 0.0604 & 0.0604 & 220 & -7.4098\\
2 & 0.044 & 0.176 & 215 & -21.8002\\
3 & 0.02 & 0.48 & 215 & -57.3657\\
4 & 0.008 &  1.536 & 196 & -143.9192\\
5 & 0.002 &  3.84 & 200 & -345.8081\\
6 & 0.0004 &  9.216 & 208 & -815.0515\\
7 & 0.0002 &  64.512 & 126 & -1879.3248\\
\end{tabular}
\end{center}
\end{table}
\end{lemma}

\begin{proof} Applying the estimation of $Q(x)$ in \eqref{eq:HbyQ} we obtain the values as shown in the table. As $\sum_{j=0}^{7}\de_j=0.185$, adding $\de_8=0.046$ we get $\de=0.231$. The found values of the $N_j$s do not exceed 220.
\end{proof}

Our aim is to prove

\begin{lemma}\label{l:dthreeprime} We have $d'''(t)<0$ for all $1\leq t \leq 2$.
\end{lemma}
\begin{proof} We approximate $d'''(t)$ by the polynomial $P_7(t)$ constructed in \eqref{eq:Pndef} as the approximate value of the order 7 Taylor polynomial of $d'''$ around $t_0:=3/2$. As the error is at most $\de$, it suffices to show that $p(t):=P_7(t)+\de<0$ in $[1,2]$. Now $P_7(1)=-0.23233...$ so $P_7(1)+\de<0$. Moreover, $p'(t)=P_7'(t)=\sum_{j=1}^{7} \dfrac{\overline{d}_j}{(j-1)!} (t-3/2)^{j-1}$ and $p'(1)=-1.411144746<0$. From the explicit formula of $p(t)$ we consecutively compute also $p''(1)=-5.536080671<0$, $p'''(1)=-16.54595998<0$ and $p^{(4)}(1)=-33.74395576<0$.

Finally, we arrive at $p^{(5)}(t)=\overline{d}_5+\overline{d}_6(t-3/2)+ (\overline{d}_7/2) (t-3/2)^{2}$. We have already checked that $p^{(j)}(1)<0$ for $j=0,1,2,3,4$, so in order to conclude $p(t)<0$ for $1\leq t\leq 2$ it suffices to show $p^{(5)}(t)<0$ in the given interval. However, the leading coefficient of $p^{(5)}$ is negative, while it is easy to see that the discriminant $\Delta:=\overline{d}_6^2-2\overline{d}_5 \overline{d}_7$ of $p^{(5)}$ is negative, too: $\Delta=-1,935,234.161$. Therefore, the whole parabola of the graph of $p^{(5)}$ lies below the $x$-axis, and so $p^{(5)}(t)<0$ for all $t \in \RR$.  It follows that also $p(t)<0$ for all $t\geq 1$. \end{proof}

And this finally proves the $k=1$ case of Conjecture \ref{conj:con3} as explained in the beginning of the section. \end{proof}


\section{The case $k=2$ of conjecture \ref{conj:con3}}\label{sec:k2case}
To show the $k=2$ case of Conjecture \ref{conj:con3} it suffices to prove Proposition \ref{prop:k2x} below.

\begin{proposition}\label{prop:k2x} Let
$F_{\pm}(x):=1+ e(x)\pm e(4x)$ and consider the $p^{\rm th}$ marginal
integrals $f_{\pm}(p):=\int_0^1 |F_{\pm}(x)|^p dx$ as well as their difference $\Delta(p):=f_{-}(p)-f_{+}(p)=\int_0^1 |F_{-}(x)|^p-|F_{+}(x)|^p dx$. Then for all $p\in(4,6)$, $\Delta(p)>0$.
\end{proposition}
\begin{proof}
As before, we put $t:=p/2\in[2,3]$ and use the notations of \eqref{eq:Gpmdef} and \eqref{eq:ddef}.
So we are to prove that $d(t)>0$ for $2<t<3$. By Parseval's formula $d(2)=d(3)=0$, now $\int_0^{1/2} G_{\pm}^2(x) dx=\int_0^1 F_{\pm}^4(x) dx=1+4+1+4+4+1=15$, and $\int_0^{1/2} G_{\pm}^3(x) dx=\int_0^1 F_{\pm}^6(x) dx=1+9+9+1+9+36+9+9+9+1=93$.

Our strategy in proving $d(t)>0$ now consists of three steps: first we prove that $d'(2)>0$, then that $d''(2)>0$, and finally that $d''$ is concave in $[2,3]$, i.e. that $d^{(4)}<0$. Since $d(2)=d(3)=0$, in view of Rolle's theorem $d'$ takes 0 at say $\tau\in (2,3)$. Since $d'(2)>0$, and $d'(\tau)=0$, in view of Lagrange's theorem $d''$ must assume some negative value. But as $d''$ is concave and $d''(2)>0$, it changes from positive to negative at a point say $\xi$. It follows that $d'$ monotonically increases between $[2, \xi]$ -- where it takes only positive values -- and then from the maximum $d'(\xi)$ it decreases between $[\xi, 3]$. As the total integral $\int_2^3 d' =0$, $d'$ eventually takes negative values, too. So $d'$ has an unique root $\theta$ in $[\xi, 3]$ and $d'$ is positive in $[\xi,\theta)$ and negative in $(\theta,3]$. So $d$ increases in $[2,\theta]$ and decreases in $[\theta,3]$ thus proving that $d(2)=d(3)=0$ are the minima of $d$ and $d>0$ in $(2,3)$.
\begin{lemma}\label{l:diffder2benpoz} We have $d'(2)>0$.
\end{lemma}
\begin{remark} By numerical calculation now $d'(2)\approx 0.03411...$ \end{remark}
Now $x$-derivatives of $G(x)$ are
\begin{align}\label{eq:GpmdefFirst2}
G_{\pm}(x)&=3+2\{\cos(2\pi x)\pm \cos(6\pi x)\pm \cos(8\pi x)\}\\
G_{\pm}^{'}(x)&=-4\pi \sin(2\pi x)\mp 12 \pi \sin(6\pi x)\mp 16 \pi \sin(8\pi x))\\
G_{\pm}^{''}(x)&=-8\pi ^2 \cos(2\pi x)\mp 72 \pi ^2 \cos(6\pi x)\mp 128 \pi ^2 \cos(8\pi x)).
\end{align}
Also we find the trivial termwise estimates
\begin{equation}\label{eq:Gpmnorm2}
\|G_{\pm}\|_\infty \leq 9, \qquad \| G_{\pm}^{'}\|_\infty \leq 32 \pi, \qquad \|G_{\pm}^{''}\|_\infty \leq 208 \pi^2 .
\end{equation}
The bound on $\| G'_{\pm}\|_{\infty}$ can slightly be improved taking into account the occurring cancellation. Namely, $G'_{\pm}(x)=-4\pi\sin(2\pi x)\left[1\pm(9-12\sin^2(2\pi x))\pm16\sqrt{1-\sin^2(2\pi x)}(1-2\sin^2(2\pi x))\right]$, so putting $v:=\sin(2\pi x)$ yields $\|G'_{\pm}\|_{\infty} \leq \max_{-1\leq v\leq 1} 4\pi |v[1\pm(9-12v^2)\pm16\sqrt{1-v^2}(1-2v^2)]|$.
Separating the cases of $G_{+}$ and $G_{-}$ and writing $w:=v^2$ we find
\begin{align*}
\|G'_{+}\|_{\infty}& \leq 8 \pi\max_{0\leq w\leq 1} \sqrt{w}||5-6w|\pm 8 \sqrt{1-w}|1-2w||\\&=8\pi\max_{0\leq w\leq 1} \sqrt{w}[|5-6w|+8\sqrt{1-w}|1-2w|]=8\pi 3.6301...
\end{align*}
and
\begin{align*}
\|G'_{+}\|_{\infty}& \leq 16 \pi\max_{0\leq w\leq 1} \sqrt{w}||2-3w|\pm 4 \sqrt{1-w}|1-2w||\\&=8\pi\max_{0\leq w\leq 1} \sqrt{w}[|2-3w|+4\sqrt{1-w}|1-2w|]=16\pi 1.6405...
\end{align*}
so in all
\begin{equation}\label{eq:improvedG'bound}
\|G'_{\pm}\|_{\infty} \leq 29.12\pi.
\end{equation}
On the other hand numerical calculation immediately gives
\begin{equation}\label{eq:Gpmminimum2}
\min_{\TT} G_{+} \approx 0.27... >1/4 \qquad {\textrm{and}}\qquad \min_{\TT} G_{-} \approx 0.063... > 1/16~\,.
\end{equation}
\
Using the notation $v:=2\pi x$ and $u:=\cos v$ as in case k=1
\begin{align*}
G'^2_{\pm}(x)&=(4\pi)^2 (\sin v\pm 3 \sin 3v \pm 4 \sin 4v)^2\\ &=(4\pi \sin v)^2 \left[1\pm  6 \cos^2 v \pm (16\cos v +3)( 2\cos^2v-1) \right]^2 \\ &= 16\pi^2(1-u^2)\left[(1\mp 3 \mp 16u \pm 12 u^2 \pm 32 u^3)\right]^2\\ &= \begin{cases} 64\pi^2 (1-u^2)(16 u^3 + 6 u^2 -8 u -1)^2 & (G=G_+)\\ 256\pi^2 (1-u^2)(-8u^3 - 3u^2+4u+1)^2 &(G=G_-)\end{cases}.
\end{align*}
On the other hand
\begin{align*}
G_{\pm}(x)&=3+2\cos v \pm 2 \cos 3v \pm 2 \cos 4v \\
&= 3+2\cos v \pm 2 (4 \cos^3 v -3 \cos v + 8 \cos^4 v - 8 \cos^2 v + 1 ) \\ &= 3 + 2 u \pm 2(8 u^4 + 4 u^3 - 8 u^2 -3 u +1)
\\& =\begin{cases} 16u^4+8u^3-16u^2-4u+5 & (G=G_+)\\ -16u^4-8u^3+16u^2+8u+1 & (G=G_-)\end{cases}.
\end{align*}
Therefore,
\begin{equation}\label{eq:GprimesquareperG2}
\frac{G'^2_{\pm}}{G_{\pm}}(x)=\begin{cases} 64\pi^2 \frac{(1-u^2)(16 u^3 + 6 u^2 -8 u -1)^2}{16u^4+8u^3-16u^2-4u+5} & (G=G_+)\\ 256\pi^2 \frac{(1-u^2)(-8u^3-3u^2+4u+1)^2}{-16u^4-8u^3+16u^2+8u+1} & (G=G_-)\end{cases}.
\end{equation}

Numerically maximizing the modulus of these rational functions in the range $u\in[-1,1]$ we obtain
\begin{equation}\label{eq:GprimecompareG2}
G'^2_{+} (x) < 2300 G_{+}(x) \qquad {\textrm{and}}\qquad G'^2_{-} (x) < 2600 G_{-}(x) .
\end{equation}

Furthermore, we analyze the function $G'' G$.
\begin{align*}
G^{''}_{\pm}(x)&=-8\pi^2(\cos v \pm 9 \cos 3v \pm 16 \cos 4v )\\
&= -8\pi^2(\cos v \pm 9 (4 \cos^3 v -3 \cos v) \pm 16(8 \cos^4 v - 8 \cos^2 v + 1 )) \\ &= -8\pi^2( u \pm 128 u^4 \pm 36 u^3 \mp 128 u^2 \mp 27u \pm 16)
\\& =\begin{cases} 128u^4+36u^3-128u^2-26u+16 & (G=G_+)\\ -128u^4-36u^3+128u^2+28u-16 & (G=G_-)\end{cases}.
\end{align*}
Therefore,
\begin{equation}\label{eq:GsecondG }
G^{''}_{\pm}G_{\pm}(x)=-16\pi^2\begin{cases} (64u^4+18u^3-64u^2-13u+8)(16u^4+8u^3-16u^2-4u+5) & (G=G_+)\\ (-64u^4-18u^3+64u^2+14u-8)(-16u^4-8u^3+16u^2+8u+1) & (G=G_-)\end{cases}.
\end{equation}
Numerically maximizing and minimizing the modulus of these functions in the range $u\in[-1,1]$ we obtain
\begin{align}\label{eq:GsecondcompareG}
  &\max_{-1\leq u\leq 1} G^{''}_{+}G_{+}(x)) < 2820{\textrm{ ,}} \qquad \min_{-1\leq u\leq 1} G^{''}_{+}G_{+}(x)) > -18500 \qquad {\textrm{and}} \\
  &\notag \max_{-1\leq u\leq 1} G^{''}_{-}G_{-}(x)) < 2710{\textrm{ ,}} \qquad \min_{-1\leq u\leq 1} G^{''}_{-}G_{-}(x)) > -14800 ,
  \qquad {\textrm{so}} \\
  &\notag \max_{-1\leq u\leq 1} G^{''}_{\pm}G_{\pm}(x)) < 2820{\textrm{ ,}} \qquad \min_{-1\leq u\leq 1} G^{''}_{\pm}G_{\pm}(x)) > -18500 .
\end{align}

\comment{We compare $G''$ and $G$, too.
\begin{align*}
G^{''}_{\pm}(x)&=-8\pi^2 (\cos v\pm 9 \cos 3v \pm 16 \cos 4v) \\ &=-8 \pi^2 \left[\cos v \pm 9 ( 4\cos^3 v-3 \cos v) \pm  16 (8\cos^4 v -8\cos^2 v +1) \right] \\ &= -8 \pi^2\left[(u \pm 128 u^4 \pm 36 u^3 \mp 128 u^2 \mp 27 u \pm 16)\right] \\ &=\begin{cases} -16\pi^2 (64 u^4+18 u^3-64u^2-13 u +8) & (G=G_+)\\ 32 \pi^2 ( 32 u^4+9u^3 - 32u^2-7u+4) & (G=G_-)\end{cases}.
\end{align*}

For the quotient we now obtain
\begin{equation}\label{eq:GdoubleprimesquareperG}
\frac{G^{''}_{\pm}}{G_{\pm}}(x)=\begin{cases} -16 \pi^2 \frac{64 u^4+18 u^3-64u^2-13 u +8}{16u^4+8u^3-16u^2-4u+5} & (G=G_+)\\ 32 \pi^2 \frac{32 u^4+9u^3-32u^2-7u+4}{-16u^4-8u^3+16u^2+8u+1} & (G=G_-)\end{cases}.
\end{equation}
So finally numerical computation yields
\begin{equation}\label{eq:GdoubleprimecompareG}
|G^{''}_{+} (x)| < 30000 G_{+}(x) \qquad {\textrm{and}}\qquad |G^{''}_{-} (x)| < 82000 G_{-}(x) .
\end{equation}
}

From \eqref{eq:Hdoubleprimegeneral} with $t=2$, $j=1$ and
estimating the norm using \eqref{eq:Gpmnorm2},
\eqref{eq:improvedG'bound}, \eqref{eq:Gpmminimum2} and
\eqref{eq:GprimecompareG2} gives
\begin{align}\label{eq:Hdoubleprimespecial2}
\|H(x)\|_\infty:&=\|H^{''}_{2,1,\pm}\|_\infty \notag \\ & =\|G''(x) G(x)\left\{2 \log G(x)+1 \right\} ~+ G'^2(x) \left\{ 2 \log G(x)  + 3 \right\}\|_{\infty} \notag \\&<\begin{cases} 2820 (2\log9 + 1) + 848 \pi^2(2\log9+3) \approx 77100... & 1\leq G(x), G''(x)>0\\ \max\{18500(2\log9+1), 848 \pi^2(2\log9+3)\}\\ < \max\{99800, 61900\}= 99800 & 1\leq G(x), G''(x)\leq0\\ 208 \pi^2(2\log{16} - 1) + 2600\cdot (2\log16 - 3) \approx  16~000...& G(x)<1 \end{cases} \\ & < 99800.\notag
\end{align}
It follows that in the numerical integration formula \eqref{eq:Riemann} the step number should be chosen to satisfy $99800/(192N^2)<0.017$, that is $N>\approx 174.86...$ i.e. $N\geq 175$. Thus the Riemann sums with $N\geq 175$ nodes will provide errors less than 0.017 in each of the two integrals $g'_{\pm}(2)$, whence the total error of $d'(2)=g_-'(2)-g_+'(2)$ must lie below 0.034 in modulus. Now the standard numerical calculation of the Rieman sums $g'_{\pm}(2)$ yields the approximate value $d'(2)\approx 0.03411...$, which exceeds $0.034$, hence the lemma is proved.
\end{proof}

\begin{lemma}\label{l:diff2der2benpoz} We have $d''(2)>0$.
\end{lemma}
\begin{remark} By numerical calculation now $d''(2)\approx 0.13757...$ \end{remark}
\begin{proof}
Now the formula \eqref{eq:Hdoubleprimegeneral} with $t=2$, $j=2$ takes the form
\begin{align}\label{eq:Hdoubleprimespecial22}
\|H''(x)& \|_\infty:=\|H^{''}_{2,2,\pm}\|_\infty \notag \\  &=\| G''(x) G(x) \log G(x)\left\{2 \log G(x)+2 \right\} ~+ G'^2(x) \left\{ 2 \log^2 G(x)  + 6 \log G(x) + 2\right\}\|_{\infty}  \notag \\&<\begin{cases} 2820\log9 (2\log9 + 2) \\+ 848 \pi^2(2\log^2 9+6\log9+2)) \approx 248000... & 1\leq G(x), G''(x)>0\\ \max\{18500\log9(2\log9+2), 848 \pi^2(2\log ^2 9+6\log9 +2)\}\\< \max\{260000, 208000\}= 260000 & 1\leq G(x), G''(x)\leq0\\ 208 \pi^2\log16(2\log{16} - 2) \\ + 2600\cdot (2\log^2 16 - 6\log16+2) \approx  27800...& G(x)<1 \end{cases} \\ & < 260000.\notag
\end{align}
In the numerical integration formula \eqref{eq:Riemann} the step number could be chosen to satisfy $\dfrac{260000}{192N^2}<0.065$, that is $N \approx 144.34 \dots$ i.e. $N\geq 145$.

The Riemann sums of the form \eqref{eq:Riemann} with $N\geq 145$ nodes will provide errors less than 0.065 in each of the two integrals $g'_{\pm}(2)$, whence the total error of $d''(2)=g_-''(2)-g_+''(2)$ must lie below 0.13. Now the standard numerical calculation of the Rieman sums $g''_{\pm}(2)$ yields the approximate value $d''(2)\approx 0.13757...$, which is well over $0.13$, hence the lemma is proved.
\end{proof}

Now we start the computation of an approximate Taylor polynomial of $d^{(4)}$.

Numerical tabulation of values give that $d^{(4)}$ is decreasing from $d^{(4)}(2)\approx -0.79041...$ to even more negative values as $t$ increases from 2 to 3. Thus our goal is to set $n\in \NN$ and $\delta_j>0$, ($j=0,\dots,n+1)$ suitably so that in the Taylor expansion
\begin{equation}\label{eq:d3Taylor_2}
d^{(4)}(t)=\sum_{j=0}^n \frac{d^{(j+4)}(\frac52)}{j!}\left(t-\frac52\right)^j +R_{n}(d^{(4)},t),\qquad
R_{n}(d^{(4)},t):=\frac{d^{(n+5)}(\xi)}{(n+1)!}\left(\xi-\frac52\right)^{n+1}
\end{equation}
the standard error estimate
\begin{equation}\label{eq:Rd4t}
|R_n(d^{(4)},t)| \leq \frac{\max_{2\leq \xi\leq 3} \|H_{\xi,n+5,+}\|_\infty + \max_{2\leq \xi\leq 3} \|H_{\xi,n+5,-}\|_\infty}{(n+1)! 2^{n+2}},
\end{equation}
calculated as in \eqref{eq:Rd3t}, provides the appropriately small error $\|R_n(d^{(4)},\cdot)\|_\infty <\delta_{n+1}$, while with appropriate approximation $\overline{d}_j$ of $d^{(j+4)}(5/2)$,
\begin{equation}\label{eq:djoverbarcriteria_2}
\left\|\frac{d^{(j+4)}(\frac52)-\overline{d}_j}{j!}\left(t-\frac52\right)^j\right\|_\infty =\frac{\left|d^{(j+4)}(\frac52)-\overline{d}_j\right|}{2^j j!}< \delta_j\qquad (j=0,1,\dots,n).
\end{equation}
Naturally, we wish to choose $n$ and the partial errors $\delta_j$ so that $\sum_{j=0}^{n+1}\de_j <\de:=0.79$, say, so that $d^{(4)}(t)< P_n(t) +\de$ with
\begin{equation}\label{eq:Pndef_2}
P_n(t):=\sum_{j=0}^n \frac{\overline{d}_j}{j!}\left(t-\frac52\right)^j.
\end{equation}
Here again we get the approximate values $\dbj$ by Riemann sums numerical integration of the integrals defining $d^{(4)}(5/2)$.

As before, for an estimation of the error we use the first formula of \eqref{eq:Riemann} with $N_j\in\NN$ steps, where $N_j$ are chosen in function of a prescribed approximation error $\eta_j$, which in turn will be set in function of the choice of $\de_j$.

So now we carry out the calculations. First, as $G_{\pm}(x)\in [1/16,9]$, $|G^{\xi}_{\pm}(x)\log^m G_{\pm}(x)|\leq \max_{[1/16,9]} |u^{\xi}\log^{m}u |$ For $m\geq 9$ the derivative of this function vanishes only at $u=1$, where the function itself vanishes, so the absolute maximum is
$$
\max\left\{\left(\dfrac{1}{16}\right)^{\xi}4^m \log^m 2,9^{\xi}2^m\log^m 3\right\}=9^\xi2^m\log^m3 \qquad \textrm{for all}\,\,  m\leq 40.
$$
In all, $\|H_{\xi,n+5,\pm}(x)\|_{\infty} \leq 9^{\xi} 2^{n+5} \log^{n+5}3 \leq 273 \cdot 2^{n+5} \log^{n+5}3$ for all $2 \leq \xi\leq 3$ and $4\leq n \leq 35$. In view of \eqref{eq:Rd4t} this yields $|R_n(d^{(4)},t)|\leq \dfrac{4368 \log^{n+5}3}{(n+1)!}<0.34$ for $n=7$.

Now we must set $\de_0,\dots,\de_7$. The goal is that the termwise error \eqref{eq:djoverbarcriteria_2} would not exceed $\de_j$, which will be guaranteed by $N_j$ step Riemann sum approximation of the two integrals defining $d^{(j+4)}(5/2)$ in \eqref{eq:djintegral}, with prescribed error $\eta_j$ each. Therefore, we set $\eta_j:=\de_j j!2^j/2$ and note that
\begin{equation}\label{eq:Njchoicek2}
N_j > N_j^{\star}:=\sqrt{\frac{\|H^{''}_{5/2,j+4,\pm}\|_\infty}{192 \eta_j}} = \sqrt{\frac{\|H^{''}_{5/2,j+4,\pm}\|_\infty}{192 j! 2^{j-1} \de_j}}
\end{equation}
 suffices.
That is, we must estimate $\|H^{''}_{5/2,j,\pm}\|_\infty$ for $j=4,\dots,11$ and thus find appropriate $N_j^{\star}$ values.

\begin{lemma}\label{l:H"norm2} For $j=4,\dots,11$ we have the following numerical estimates for the values of $\|H^{''}_{5/2,j,\pm}\|_\infty$.
\end{lemma}

\begin{table}[H]
\caption{Estimated values of $\|H^{''}_{5/2,j,\pm}\|_\infty$ for $j=4,\dots,11$.}
\label{table:Hdoubleprimenorm2}
\begin{center}
\begin{tabular}{cc}
$j$ \qquad & estimate for $\|H^{''}_{5/2,j,\pm}\|_\infty$ \\
4 &  16,000,000 \\
5 & 40,000,000 \\
6 & 104,000,000\\
7 & 267,000,000 \\
8 & 680,000,000 \\
9 & 1,705,000,000 \\
10 & 4,255,000,000 \\
11 & 10,600,000,000 \\
\end{tabular}
\end{center}
\end{table}

\begin{proof} Whether we consider $H_+$ or $H_-$, the range of $G_{\pm}$ stays in $[1/16,9]$, so
$$
\|H^{''}_{\pm}\|_\infty \leq \max\{\max_A |H^{''}_{5/2,j,\pm}|,~ \max_B |H^{''}_{5/2,j,\pm}| \}
$$
with $A:=\{x\in[0,1/2]~:~ G_{\pm}(x)\leq 1\}$ and $B:=\{x\in[0,1/2]~:~ G_{\pm}(x)>1\}$. Recall that from \eqref{eq:Hdoubleprimegeneral} we find for arbitrary $j\geq 4$ and with $G=G_{\pm}$
\begin{align*}
H^{''}_{5/2,j,\pm}= G'' G^{3/2} \log^{j-1}G \left(\frac52\log G +j\right) + G'^2\sqrt{G} \log^{j-2}G\left(\frac{15}{4}\log^2 G +4 j \log G + j(j-1)\right).
\end{align*}
Since we have no control over the sign of $G''$, we now estimate trivially -- using \eqref{eq:Gpmnorm2} -- as

\begin{align}\label{eq:HbyQk2}
|H^{''}_{\pm}(x)| & \leq \|G''\| G^{3/2}(x) \left|\log^{j-1}G(x) \right| \left|\frac52\log G(x) +j\right| + Q(x) \notag \\ &\left( \textrm{with} \quad Q(x):=G'^2(x)\sqrt{G(x)} \left|\log^{j-2}G(x)\right|\left|\frac{15}4\log^2 G(x) +4j \log G(x) + j(j-1)\right| \right) \notag \\ &\leq 208 \pi^2 27  \log^{j-1}9 (j+\log 243)+ Q(x).
\end{align}
Now by the two estimates of $|G'(x)|$ from \eqref{eq:Gpmnorm2} $32 \pi$ and from \eqref{eq:GprimecompareG2} $\sqrt{2600 G(x)}$ it follows that
$$
Q(x)\leq \begin{cases} 2600 G(x)^{3/2} \left|\log^{j-2}G(x)\right|\left|\frac{15}4\log^2 G(x) +4j \log G(x) + j(j-1)\right| & x\in A \\ 1024 \pi^2 G(x)^{1/2} \log^{j-2}G(x) \left(\frac{15}4\log^2 G(x) +4j \log G(x) + j(j-1)\right)  & x\in B
\end{cases}.
$$
Now observe that here $Q$ is estimated by functions of $G(x)$, so we can look for maximization or good estimates on the range of $G$. For $x\in A$ denote $u:=- \log G(x)$: then the condition $x\in A$ means that $0\leq u < \log 16$, while for $x\in B$ the substitution $u:=\log G(x)$ leads to $0<u \leq \log 9$.

In all we find with
$$
\psi(u):=2600 e^{-\frac32 u} u^{j-2} \left| \frac{15}4 u^2 - 4j u + j(j-1)\right|
$$
and
$$
\varphi(u) :=1024 \pi^2 e^{u/2} u^{j-2} \left( \frac{15}4 u^2 + 4j u + j(j-1)\right)
$$
that
\begin{align*}
\| Q\|_\infty & \leq \max\left\{ \max_{0\leq u \leq \log 16} \psi(u), \max_{0\leq u \leq \log 9} \varphi(u) \right\}\\ &= \max\left\{ \max_{\log 9 \leq u \leq \log 16} \psi(u), \max_{0\leq u \leq \log 9} \max\{\psi(u),\varphi(u)\} \right\}.
\end{align*}
Now it is clear that for any real $u\geq 0$ we have $| \frac{15}4 u^2 - 4j u + j(j-1)| \leq \frac{15}4 u^2 + 4ju + j(j-1)$ and $e^{u/2}>e^{-\frac32 u}$, whence in view of $2600<1024 \pi^2$, necessarily $\psi(u)\leq \varphi(u)$. Therefore, in the $[0,\log 9]$ interval $\max\{\psi(u),\varphi(u)\}=\varphi(u)$. Observe that $\varphi$ is increasing in $[0,\infty)$, whence
$$
\max_{0\leq u \leq \log 9} \varphi(u)=\varphi(\log 9)= 1024\pi^2 \cdot 3 \left\{\frac{15}{4}\log^j9 + 4j\log^{j-1} 9 + j(j-1)\log^{j-2} 9 \right\}.
$$
Turning to the expression with $\psi$, trivially estimating it gives
$$
\psi(u)\leq S:=2600\cdot \dfrac1{27} \cdot \left\{ \dfrac{15}{4}\log^{j} 16 + 4j \log^{j-1} 16 + j(j-1)\log^{j-2} 16 \right\}.
$$
Comparing termwise, we easily see that $1024\pi^2\times3\times \log^m 9 \geq (2600/27)\times \log^m 16$ for $m=j-2$, $m=j-1$ and $m=j$ and $j=4,\dots,11$, that is, for $m=2,\dots,11$.
Indeed, this is equivalent to $1024\pi^2\times 27/2600 \geq (\log 16 /\log 9) ^m$, that is $m \log \left(\dfrac{\log 4}{\log 3}\right) \leq \log \{1024\pi^2\times 27/2600\}$, which holds true even up to $m=20$. Therefore, for $j=4,\dots,11\,$ $S< \varphi(\log 3)$.

In all, $\|Q\|_\infty \leq \varphi(\log 3) = 1024\pi^2 \cdot 3 \left\{\frac{15}{4}\log^j9 + 4j\log^{j-1} 9 + j(j-1)\log^{j-2} 9 \right\}$.

So for $j=4,5,6,7,8,9,10,11$ we get the values 9,552,472...,
26,388,587.4..., 71,259,267.7..., 188,851,956.3...,
492,698,473.2..., 1,268,392,131.6..., 3,228,178,830...,
8,134,909,871.3... respectively for $\|Q\|_\infty$. Adding
$208\pi^2\times27\log^{j-1}9(j+\log 243)$, \eqref{eq:HbyQk2}
yields the listed values of Table \ref{table:Hdoubleprimenorm2}.
\end{proof}

\begin{lemma}\label{l:d3at52} Set $\de_j$ as $\de_0=0.3$, $\de_1=0.08$, $\de_2=0.03$, $\de_3=\de_4= \de_5=0.01$, $\de_6=\de_7=0.005$. Then the approximate Riemann sums of order $N_j$ yield the approximate values $\overline{d}_j$ as listed in Table \ref{t:Deltas_2}, admitting the error estimates \eqref{eq:djoverbarcriteria_2} for $j=0,\dots,7$. Furthermore, $\|R_{7}(d^{(4)},t)\|_{\infty} <0.34=:\de_8$ and thus with the approximate Taylor polynomial $P_7(t)$ defined in \eqref{eq:Pndef_2} the approximation $|d^{(4)}(t)-P_7(t)|<\de=0.79$ holds uniformly for $2\leq t \leq 3$.
\end{lemma}
\begin{table}[H]
\caption{The chosen values of $\de_j$, $\eta_j$, the appropriate $N_j$ and approximate Taylor coefficients}
\label{t:Deltas_2}
\begin{center}
\begin{tabular}{|c|c|c|c|c|c|}
$j$ & $\de_j$ & $\eta_j$ & $N_j$ & $\overline{d}_j$\\
0 & 0.13 &  0.065 & 1102 & -8.4790\\
1 & 0.15 &  0.150 & 1178 & -31.5452\\
2 & 0.1 & 0.4 & 1164 & -99.8194\\
3 & 0.05 & 1.2 & 1077 & -287.2717\\
4 & 0.015 & 2.88 & 1107 & -776.5678\\
5 & 0.004 &  7.68 & 1076 & -2010.9552\\
6 & 0.0008 &  18.432 & 1097 & -5043.6133\\
7 & 0.0002 &  64.512 & 923 & -12356.378\\
\end{tabular}
\end{center}
\end{table}

\begin{proof} Applying the estimation of $Q(x)$ in \eqref{eq:HbyQk2} we obtain the values as shown in the table. As $\sum_{j=0}^{7}\de_j=0.45$, adding $\de_8=0.34$ we get $\de=0.79$. The found values of the $N_j$s do not exceed 1200.
\end{proof}

Our aim is to prove
\begin{lemma}\label{l:dthreeprime2} We have $d^{(4)}(t)<0$ for all $2\leq t \leq 3$.
\end{lemma}
\begin{proof} We approximate $d^{(4)}(t)$ by the polynomial $P_7(t)$ constructed in \eqref{eq:Pndef_2} as the approximate value of the order 7 Taylor polynomial of $d^{(4)}$ around $t_0:=5/2$. As the error is at most $\de$, it suffices to show that $p(t):=P_7(t)+\de<0$ in $[2,3]$. Now $P_7(2)=-0.79075...$ so $P_7(2)+\de<0$. Moreover, $p'(t)=P_7'(t)=\sum_{j=1}^{7} \dfrac{\overline{d}_j}{(j-1)!} (t-5/2)^{j-1}$ and $p'(2)=-5.557576563... <0$. From the explicit formula of $p(t)$ we consecutively compute also $p''(2)=-21.27623445...<0$, $p'''(2)=-77.45997012...<0$ and $p^{(4)}(2)=-144.1173211...<0$.

Finally, we arrive at $p^{(5)}(t)=\overline{d}_5+\overline{d}_6(t-5/2)+ (\overline{d}_7/2) (t-5/2)^{2}$. We have already checked that $p^{(j)}(2)<0$ for $j=0,1,2,3,4$, so in order to conclude $p(t)<0$ for $2\leq t\leq 3$ it suffices to show $p^{(5)}(t)<0$ in the given interval. However, the leading coefficient of $p^{(5)}$ is negative, while it is easy to see that the discriminant $\Delta:=\overline{d}_6^2-2\overline{d}_5 \overline{d}_7$ of $p^{(5)}$ is negative, too: $\Delta\approx -24,258,211$. Therefore, the whole parabola of the graph of $p^{(5)}$ lies below the $x$-axis, and so $p^{(5)}(t)<0$ for all $t \in \RR$.  It follows that also $p(t)<0$ for all $t\geq 2$. \end{proof}

And this finally proves the $k=2$ case of Conjecture \ref{conj:con3} as explained in the beginning of the section.

\section{Final remarks}\label{sec:final}

We have encountered no theoretical difficulties in calculating the above cases, and it seems that a similar numerical analysis should work even for larger $k$. In case the errors and step numbers would grow, we could as well apply Taylor expansion around more points, say around $t_0:=k+1/4$ and $s_0:=k+3/4$, which reduces the radius from $1/2$ to $1/4$. So in principle a numerical analysis is possible.

Numerical tabulation of the functions $d(t)$ in various ranges $[k,k+1]$ have led to similar pictures for $k=3,4,...13$. We tabulated the difference function $d(t)=d_k(t)$ for $k$ up to 13, and found the difference function to be positive in all cases. That of course suggests that Conjecture \ref{conj:con3} holds true.

When writing $t=k+s$, where now $0\leq s\leq 1$, and after normalizing say by the maximum value, it seems that the shapes of $f_k:=f_k(s):=d(k+s)/\max_{[k,k+1]} d$ approach a fine mathematical curve, something quite resembling to a reflected log-normal distribution density function shape, having maximum somewhere at $s_0\approx 0.85$. Perhaps the limit distribution, i.e. $f(s):=\lim_{k\to \infty} f_k(s)$ can be found, and thus at least in the limit we can derive positivity of the function $d(t)$.

Computation of Taylor coefficients at the center-points, that is derivatives of the difference function $d(t)$ at $t=k+1/2$ led to the unexpected finding that the Taylor coefficients $d_j:=d^{(k+2+j)}(k+1/2)$ of $d^{(k+2)}$ remained of constant negative sign. Without deriving precise error estimates, we continued the calculation of the approximative value $\overline{d}_j$ of these Taylor coefficients for various further $j$ and for some higher $k$, finding in all studied cases that $\overline{d_j}<0$. Also, the phenomenon, which helped us to execute theoretically precise proofs, that $d^{(j)}(t)<0$ for some $j=j(k)$, seems to remain in effect also for higher $k$ and at least for $j=k+2$. A theoretically precise proof of these facts would ease considerably the proof of validity of Conjecture \ref{conj:con3}.

Also we tested the "Hardy-Littlewood case" of Conjecture \ref{conj:con3}, that is,
$t=k+1/2$, i.e. $p=2k+1$, which was the original example of Hardy and Littlewood in case $k=1$. Up to $k=14$, we found positive, though decreasing numerical values. However, it is quite strange that the integrals of $G_{\pm}^t$ increase (close to $10^{11}$ when $k=13$), yet the found difference is smaller and smaller (of the order $10^{-3}$ when $k$ reaches 13). The relative size of the difference is thus found to be some $10^{-15}$ times the size of the individual integrals, which suggests that choice of the step size ($10^{-3}$ in our case) in the Riemann sum and errors in the computation of the respective integrals amount much higher quantities than the found values of the difference. Clearly when coming closer say to the left endpoint $t=k$, the difference can be even smaller. Therefore, these numerical experiments are far from mathematically reliable.

\begin{figure}
\centering
\includegraphics[width=8cm]{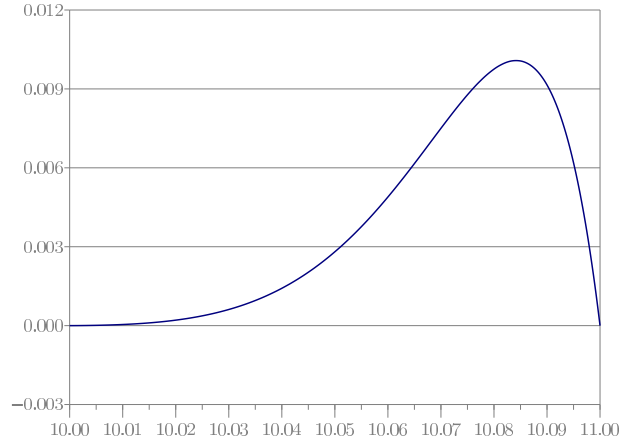}
\caption{A typical graph of $d(t)$: here $k=10$, $10\leq t\leq 11$. The step size used in the Riemann sums is 0.001, the tabulation is with density 0.01.}
\end{figure}


\end{document}